\documentclass[reqno]{amsart}

\usepackage{amsmath,amssymb,amscd, accents} \usepackage{graphicx}
\DeclareGraphicsExtensions{.eps} \usepackage{mathrsfs}
\usepackage[mathcal]{eucal}

\newtheorem{Thm}{Theorem}{\bfseries}{\itshape}
\newtheorem*{Thm*}{Theorem}{\bfseries}{\itshape}
\newtheorem{Cor}{Corollary}{\bfseries}{\itshape}
\newtheorem{Prop}[Cor]{Proposition}{\bfseries}{\itshape}
\newtheorem{Lem}[Cor]{Lemma}{\bfseries}{\itshape}
\newtheorem*{Lem*}{Lemma}{\bfseries}{\itshape}
\newtheorem{Fact}[Cor]{Fact}{\bfseries}{\itshape}
{\bfseries}{\itshape}
\newtheorem{Def}[Cor]{Definition}{\bfseries}{\rmfamily}
{\scshape}{\rmfamily}
\newtheorem{Rem}[Cor]{Remark}{\scshape}{\rmfamily}

\renewcommand\ge{\geqslant} \renewcommand\le{\leqslant}
\let\tildeaccent=\~ \let\hataccent=\^
\renewcommand\~[1]{\widetilde{#1}}

\def\<{\left<} \def\>{\right>} \def\({\left(} \def\){\right)}

\def\abs#1{\left\vert #1 \right\vert} \def\norm#1{\left\Vert #1
  \right\Vert} 

\let\parasymbol=\S \def\secref#1{\parasymbol\ref{#1}}
 \def\pd#1#2{\tfrac{\partial#1}{\partial#2}}

\let\polishL=l \def\Zoladek.{\.Zol\c adek}

 \def\const{\operatorname{const}}
\def\codim{\operatorname{codim}}

 \def\etc.{\emph{etc}.}
 \def\Sing{\operatorname{Sing}}

\def\:{\colon} \def\R{{\mathbb R}} \def\C{{\mathbb C}} \def\Z{{\mathbb
    Z}} \def\N{{\mathbb N}}  
 \def\e{\varepsilon} \def\S{\varSigma}
\def\l{\lambda}   
 
 \def\d{\,\mathrm d}

 \let\PolishL=\L \def\Lojas.{\PolishL ojasiewicz}

 \def\cL{{\mathcal L}}

\def\cO{{\mathcal O}}

 \def\mult{\operatorname{mult}}

\def\rest#1{{\vert_{#1}}}

\def\clo{\operatorname{Clo}}
\def\vol{\operatorname{Vol}}
\def\supp{\operatorname{supp}}
\def\vell{{\boldsymbol\ell}}

\def\qmi{W^s}

\newcommand{\pv}[1][k]{\Gamma^{#1}_{\vell}}
\newcommand{\tpv}[1][k]{\tilde\Gamma^{#1}_{\vell}}
\newcommand{\bpv}[1][k]{\bar\Gamma^{#1}_{\vell}}
\newcommand{\las}[1][k]{L^{#1}_{\vell}}
\newcommand{\mc}{\mathcal{M}}

\def\ot#1#2{\mathrel{\overset{\text{#2}}{#1}}}
\def\degf{\mathfrak{D}}

\begin{document}

\title{Multiplicity estimates: a Morse-theoretic approach}

\author{Gal Binyamini}\address{University of Toronto, Toronto, 
Canada}\email{galbin@gmail.com}
\thanks{The author was supported by the Banting
  Postdoctoral Fellowship and the Rothschild Fellowship}

\begin{abstract}
  The problem of estimating the multiplicity of the zero of a
  polynomial when restricted to the trajectory of a non-singular
  polynomial vector field, at one or several points, has been
  considered by authors in several different fields. The two best
  (incomparable) estimates are due to Gabrielov and Nesterenko.

  In this paper we present a refinement of Gabrielov's method which
  simultaneously improves these two estimates. Moreover, we give a
  geometric description of the multiplicity function in terms of
  certain naturally associated \emph{polar varieties}, giving a
  topological explanation for an asymptotic phenomenon that was
  previously obtained by elimination theoretic methods in the works of
  Brownawell, Masser and Nesterenko. We also give estimates in terms
  of Newton polytopes, strongly generalizing the classical estimates.
\end{abstract}
\maketitle
\date{\today}

\section{Introduction}
\label{sec:intro}

Consider a polynomial vector field $V$ of degree $\delta$ and a polynomial
$P$ of degree $d$ on $\C^n$,
\begin{equation} \label{eq:VP}
  \begin{split}
    V = \sum_{i=1}^n Q_i \pd{}{x_i} \qquad & Q_i\in\C[x_1,\ldots,x_n],\ \deg Q_i\le\delta \\
                                          & P\in\C[x_1,\ldots,x_n],\ \deg P\le d.
  \end{split}
\end{equation}
For $p\in\C^n$ a non-singular point of $V$, let $\gamma_p$ denote the
germ of the trajectory of $V$ through the point $p$. We define
$\mult_p^V P$, the \emph{multiplicity of $P$ at the point $p$}, to be
the order of zero of $P$ along $\gamma_p$. Alternatively, we may think of $V$
as a general non-linear system of differential equations of the time
variable $t$, the variables $x_1,\ldots,x_n$ as the dependent
variables, and $\mult_p^V P$ as the order of zero of $P$ evaluated on
a particular solution of the system.

A multiplicity estimate is an answer for the following question:
\emph{For a given $P$ and $V$, how large can $\mult_p^V P$ be?}. It
is usual for the answer to be given in terms of the dimension $n$ and
the degrees $d,\delta$. More generally, given a finite set of points
$p_1,\ldots,p_\nu$ one may ask for an upper bound for
$\sum_i\mult^V_{p_i}P$ or for $\min_i \mult^V_{p_i} P$, depending on the
geometry of the set.


\subsection{Historical sketch}
\label{sec:history}

The problem of deriving multiplicity estimates of the types mentioned
above has been considered in various areas of mathematics. Below we
give a brief outline of some of the main contributions. Our
presentation follows the development in each field separately, rather
than in historical order.

In transcendental number theory, the subject began with Nesterenko's
contributions to the Siegel-Shidlovski theory of
E-functions~\cite{nesterenko:e-funcs}. Nesterenko developed an
algebraic technique for the estimation of multiplicities for functions
satisfying a certain type of linear differential equations. Similar
ideas were later successfully applied to non-linear systems by
Brownawell~\cite{brownawell:zero-est,brownawell:zero-ord}, by
Brownawell and Masser~\cite{bm:mult-I,bm:mult-II} and by
Nesterenko~\cite{nesterenko:mult-estimates}; these authors also gave
estimates for the sum of multiplicities over an arbitrary finite set
of points, which is of importance in transcendental number theory. The
algebraic techniques developed in these papers were later applied to
various more specific situations, notably to translation-invariant
vector fields on group varieties and sets of points related to the
group structure, leading to great progress in the field
(see~\cite{masser:zero-est-survey} for a survey).

In control theory, Risler~\cite{risler:nonholonomy} suggested the
problem of multiplicity estimates in the study of nonholonomic control
systems. Risler considered the planar case $n=2$, obtained a
multiplicity estimate, and used it to bound the degree of nonholonomy
for planar control systems. Gabrielov and Risler extended this work
and gave good multiplicity estimates for the case $n=3$ in
\cite{gr:mult-c3}. For arbitrary dimension, Gabrielov obtained a
multiplicity estimate in~\cite{gabrielov:mult-old}, and subsequently
developed a powerful technique involving Milnor fibers of certain
deformations, leading to much sharper estimates
in~\cite{gabrielov:mult}. Khovanskii later simplified Gabrielov's
arguments considerably using the notion of integration over Euler
characteristics. This simplification was later applied by Gabrielov
and Khovanskii to establish multiplicity estimates in the
multi-dimensional setting in~\cite{gk:mult}.

In the qualitative theory of differential equations, Novikov and
Yakovenko~\cite{ny:chains} have obtained multiplicity estimates in
their study of abelian integrals, relatsed to the infinitesimal Hilbert
16th problem. While not as sharp as the preceding estimates, these
estimates remain valid when one considers the number of zeros in a
small interval of prescribed length. Yomdin~\cite{yomdin:oscillation}
has studied a similar problem in the context of bifurcations of zeros
in analytic families, and obtained upper bounds using Gabrielov's
estimate.

\subsection{The estimates of Nesterenko and Gabrielov}
\label{sec:background}

In this section we give precise statements and some further discussion of
the multiplicity estimates of Nesterenko and Gabrielov, which are the
two best known estimates in our context. In summary, these two
estimates are incomparable: Nesterenko's estimate is sharp up to a
multiplicative constant with respect to the degree $d$, which is the
main asymptotic considered in transcendental number theory, but
doubly-exponential in $n$; Gabrielov's estimate on the other hand is
not sharp with respect to $d$, but exhibits an essentially optimal
simply-exponential growth with respect to the dimension $n$.

\subsubsection{Nesterenko's estimate}
\label{sec:nesterenko}

The estimates presented in this subsection are those
of~\cite{nesterenko:mult-estimates}. This work improved several
previous results by
Brownawell~\cite{brownawell:zero-est,brownawell:zero-ord} and by
Brownawell and Masser~\cite{bm:mult-I,bm:mult-II}. Brownawell and
Masser's principal idea was that while the multiplicity at a given
point may be quite large, this cannot occur too frequently. If one
sums up the multiplicities over several points, most points will
contribute terms of lower order. Nesterenko establishes a more refined
result following the same paradigm.

Nesterenko states his results in the projective context, for
homogeneous vector fields. To make the comparison with the rest of our
text transparent, we translate his result to the affine context. The
two formulations are easily seen to be equivalent (up to the precise
values of constants).

Let $V,P$ be as in~\eqref{eq:VP}, and let $p_1,\ldots,p_\nu\in\C^n$ be
non-singular points of $V$ which belong to the same trajectory of $V$
(that is, $\gamma_{p_i}$ can be obtained by analytic continuation
from $\gamma_{p_j}$ for any $1\le i,j\le\nu$), and assume that $P$
does not vanish identically on this trajectory. Let $\kappa$ denote
the transcendence degree of this trajectory (i.e., the dimension of the
smallest algebraic set containing the trajectory).

\begin{Thm*}[\protect{\cite[Theorem 1]{nesterenko:mult-estimates}}]
  For any trajectory $\gamma$ as above there exists a constant
  $C_\gamma$ such that for any collection of points $p_1,\ldots,p_\nu$
  belonging to $\gamma$,
  \begin{equation}
    \sum_{i=1}^\nu \mult_{p_i}^V P \le C_\gamma \sum_{j=1}^\kappa a_{\kappa-j}(C_\gamma d^j) d^j    
  \end{equation}
  where $a_q(T)$ denotes the maximum number of points among
  $p_1,\ldots,p_\nu$ lying in an irreducible variety of dimension $q$
  and degree at most $T$ in $\C^n$. In particular $a_0(T)\equiv1$.
\end{Thm*}

The constant $C_\gamma$ is not explicitly worked out
in~\cite{nesterenko:mult-estimates}, but from the proof one can
determine that it grows doubly-exponentially with the dimension $n$.
We note also that when $\kappa<n$, i.e. the trajectory is not
completely transcendental, the constant $C_\gamma$ depends on
algebraic complexity (for instance, the degree) of the Zariski closure
of $\gamma$. In this sense, the estimate is not entirely explicit,
since one cannot in general estimate the degree of this Zariski
closure purely in terms of $n,d,\delta$. For a discussion and a
comparison of Nesterenko's result and our result in this context see
Remark~\ref{rem:small-kappa}.

A few remarks are in order. If we restrict our attention to the case
of a single point $p$ and make no special assumptions on
$\kappa$\footnote{The situation with $\kappa<n$ is similar}, then the
theorem states that there exists a constant $C_p$ such that
$\mult_p^V P \le C_p d^n$. Since the linear space of polynomials of
degree $d$ has dimension of the order of $d^n$, this result is
essentially the best possible up to a multiplicative constant.

One could naively estimate the sum of the multiplicities over $\nu$
different points by $\nu d^n$. However, Nesterenko's result implies
that the coefficient of the $d^n$ term is a constant independent of
the number and position of the points $p_1,\ldots,p_\nu$. The
coefficient of the next term, of order $d^{n-1}$, may already depend
on the number and position of the points. However, the theorem
essentially states that the number of contributions of this order is
bounded by the number of points $p_i$ that could belong to an
irreducible curve of degree $C_\gamma d^{n-1}$. Next we have a contributions
of order $d^{n-2}$, whose number is bounded by the number of points
$p_i$ that could belong to an irreducible surface of degree $C_\gamma d^{n-2}$,
and so on.

As we have already seen, even for the case of a single point, the
growth of $\mult_p^V P$ with respect to $n$ is at least exponential in
$n$. However, the dependence of the constant $C_\gamma$ on the dimension $n$
is doubly exponential. In this sense the result is possibly
suboptimal, and as we shall see in~\secref{sec:gabrielov}, the
multiplicity can in fact grow no faster than exponentially in $n$.

\subsubsection{Gabrielov's estimate}
\label{sec:gabrielov}

We turn now to Gabrielov's estimate presented in~\cite{gabrielov:mult}.
In this work only the case of a single point was considered. Therefore
let $V,P$ be as in~\eqref{eq:VP} and $p\in\C^n$ a non-singular point
of $V$. Assume that $P$ does not vanish identically on $\gamma_p$ (the
trajectory through $p$).

\begin{Thm*}[\protect{\cite[Theorem 2]{gabrielov:mult}}]
  We have the upper bound
  \begin{equation}
    \mult_p^V P \le 2^{2n-1} \sum_{i=1}^n \left[ d+(i-1)(\delta-1)\right]^{2n}
  \end{equation}
\end{Thm*}

The dependence of this estimate on $n$ is simply-exponential, which as
we have seen is essentially the best possible. However, with respect
to $d$ this estimate has order $d^{2n}$, which is the square of the
correct growth (as we know from~\secref{sec:nesterenko}).

\subsection{Overview of this paper}

In this paper we consider the problem of bounding the multiplicity at
a point $p$, and more generally, the sum of multiplicities over an
arbitrary finite set of points. Our approach is based on a refinement
of Gabrielov's deformation technique. Following Gabrielov's ideas, we
translate the problem of estimating multiplicities to the problem of
estimating the Euler characteristics of Milnor fibers of certain
deformations related to $P$ and $V$. More generally, we consider the
problem of estimating the individual Betti numbers of Milnor fibers of
general deformations (under a certain smoothness assumption).

Through classical techniques of polar varieties, we translate the
problem of estimating Betti numbers to the study of some naturally
associated algebraic cycles and their intersection numbers. One can
then apply ideas from algebraic geometry to estimate the degrees of
the cycles and, consequently, obtain upper bounds for their
intersection numbers. Moreover, the algebraic cycles are defined
globally and provide a clear geometric picture for the situation
involving several points $p_1,\ldots,p_\nu$. As a result we obtain a
multiplicity estimate which simultaneously improves the estimates of
Nesterenko and of Gabrielov.

Our basic multiplicity estimate, for a single point $p$ and in terms
of the parameters $n,d,\delta$ in~\eqref{eq:VP} is the following
direct corollary of Theorem~\ref{thm:mult-mc-uniform-bound} and
Proposition~\ref{prop:mc-degree-affine}.

\begin{Cor}
  Let $d\ge n-1$ and let $p\in\C^n$ be a non-singular point
  of $V$, and assume that $\mult_p^V P$ is finite. Then
  \begin{equation}
    \mult_p^V P \le 2^{n+1} (d+(n-1)(\delta-1))^n 
  \end{equation}
\end{Cor}

Since our description is given in terms of certain naturally defined
varieties, it can be adapted to take into account additional geometric
structure on the ambient space, the polynomial $P$ and the vector
field $V$. To demonstrate this, we give multiplicity estimates
depending on the volumes of the Newton polytopes of $P$ and $V$ in the
context of the torus group $(\C^*)^n$, analogous to the BKK theorem
(which is reviewed in~\secref{sec:bkt}). The basic estimate in terms
of degrees follows immediately as a special case.

Below $\Delta(P),\Delta(V)$ denote the Newton polytopes of $P$ and $V$
respectively, and $\Delta_x$ denotes the standard simplex in the
$x$-variables (see~\secref{sec:bkt} for the definitions). The
following is a direct corollary of
Theorem~\ref{thm:mult-mc-uniform-bound} and
Proposition~\ref{prop:mc-degree}.

\begin{Cor}
  Let $(n-1)\Delta_x\subset\Delta(P)$ and let $p\in(\C^*)^n$ be a non-singular point
  of $V$, and assume that $\mult_p^V P$ is finite. Then
  \begin{equation}
    \mult_p^V P \le 2^{n+1} n! \vol(\Delta(P)+(n-1)\Delta(V)+\Delta_x) 
  \end{equation}
\end{Cor}

We now present two analogous multiplicity estimates describing the
behavior of the multiplicity function as the point $p$ varies. We
begin with a definition.

\begin{Def} \label{def:degf}
  Let $M$ denote $\C^n$ or $(\C^*)^n$.
  For an irreducible variety $W\subset M$, we define the function
  \begin{equation}
    \degf_W:M\to\N \qquad \degf_W(p) = \begin{cases}
      \deg W & p\in W \\
      0 & \text{otherwise}
    \end{cases}
  \end{equation}
  We extend this by linearity to define $\degf_C$ for an arbitrary
  effective algebraic cycle $C\subset M$.
\end{Def}

Once again we give two statements, the former in terms of the degrees
$d,\delta$ and the latter in terms of the Newton polytopes
$\Delta(P),\Delta(V)$. These two results are direct corollaries of
Theorem~\ref{thm:mult-mc-uniform-bound} and
Propositions~\ref{prop:mc-degree} and~\ref{prop:mc-degree-affine}.

\begin{Cor}
  Let $d\ge n-1$. There exist algebraic cycles
  \begin{equation}
    \mc^0(P),\ldots,\mc^{n-1}(P)\subset\C^n
  \end{equation}
  where
  \begin{align*}
    \dim\mc^k(P) &= k \\
    \deg\mc^k(P) &< 2^n(d+(n-k-1)(\delta-1))^{n-k}
  \end{align*}
  such that for every $p\in\C^n$ where $V$ is non-singular and
  $\mult_p^V P$ is finite,
  \begin{equation} \label{eq:thm-simple}
    \mult_p^V P \le \sum_{k=0}^{n-1} \degf_{\mc^k(P)}(p).
  \end{equation}
\end{Cor}

Below $\qmi_k(\cdot)$ denotes the $k$-th simplicial quermassintegral
introduced in~\eqref{eq:qmi-def}.

\begin{Cor}
  Let $(n-1)\Delta_x\subset\Delta(P)$. There exist algebraic cycles
  \begin{equation}
    \mc^0(P),\ldots,\mc^{n-1}(P)\subset(\C^*)^n
  \end{equation}
  where
  \begin{align*}
    \dim\mc^k(P) &= k \\
    \deg\mc^k(P) &< 2^n n! \qmi_k(\Delta(P)+(n-k-1)\Delta(V)+\Delta_x)
  \end{align*}
  such that for every $p\in\C^n$ where $V$ is non-singular and
  $\mult_p^V P$ is finite,
  \begin{equation} 
    \mult_p^V P \le \sum_{k=0}^{n-1} \degf_{\mc^k(P)}(p).
  \end{equation}
\end{Cor}

The two corollaries above present a picture similar to the one given
in Nesterenko's estimate. In~\secref{sec:improving-ng} we show how to
derive Nesterenko's result from ours.

We remark that in the full formulation of our result, the contribution
of each cycle $\mc^k(P)$ is not the degree $\degf_{\mc^k(P)}$, but
rather the order of intersection between $\mc^k(P)$ and a certain
special affine-linear space passing through $p$. While we do not fully
investigate this in the present paper, in some contexts this extra
information can lead to significantly stronger estimates than the
naive one used in~\eqref{eq:thm-simple}.

\subsection{Contents of the paper}

The contents of this paper are as follows. In~\secref{sec:prelims} we
discuss general preliminaries. In~\secref{sec:milnor-fibers} we
discuss Milnor fibers and their relation to multiplicity estimates.
In~\secref{sec:betti} we give estimates for the Betti numbers of the
Milnor fiber. In~\secref{sec:mult-estimates} we give the full
formulation of our multiplicity estimates in various contexts. We show
how our estimates improve those of Nesterenko, Gabrielov and Risler.

In Appendix~\ref{sec:appendix} we discuss a general compactness
property which is useful in establishing uniform algebraic
semicontinuous bounds. In Appendix~\ref{sec:notations} we give
a list of the main notations used in this paper.

\subsection{Acknowledgements}

I would like to express my gratitude to Askold Khovanskii, Andrei
Gabrielov and David Massey for invaluable discussions during the
preparation of this manuscript. I also wish to thank the anonymous
referees for many suggestions improving the accuracy and readability
of the text.

\section{General preliminaries}
\label{sec:prelims}

In this section we discuss some preliminaries that shall be
needed in the sequel. In~\secref{sec:cycles} we review the
theory of algebraic cycles and their intersection numbers.
In~\secref{sec:bkt} we review the notion of mixed volume of
convex bodies and its relation to intersection theory on the
torus $(\C^*)^n$ through the Bernstein-Kushnirenko-Khovanskii theorem.

\subsection{Algebraic cycles and intersection numbers}
\label{sec:cycles}

We introduce some basic results on algebraic cycles and their
intersection theory. For the purposes of this paper we will assume
that the ambient variety $M$ is given by $\C^n$ or $(\C^*)^n$ in the
algebraic case, or by the germ of these varieties at a point in the
analytic case (although the subject can be developed in far greater
generality, see~\cite{fulton:it} for a canonical reference). We denote
the coordinate ring of $M$ by $R$.

A \emph{$k$-cycle} is a finite formal sum $\sum n_i [V_i]$ where
$V_i\subset M$ are $k$-dimensional irreducible subvarieties of $M$ and
$n_i$ are integers. A \emph{cycle} is a (finite) sum of cycles of any
dimension. In this paper, we shall deal exclusively with cycles with
positive coefficients.

We say that two varieties $V,W\subset M$ intersect properly at a
component $Z\subset V\cap W$ if $\codim Z=\codim V+\codim W$. In this
case there is a well defined intersection number $i(Z;V\cdot W;M)$.
If $V,W$ intersect properly at every component of their intersection
then there is a well defined intersection cycle
\begin{equation}
  V\cdot W = \sum_{Z\subset V\cap W} i(Z;V\cdot W;M) [Z].
\end{equation}
This product can be extended by linearity to the product of arbitrary
cycles, assuming that all intersections are proper.

We now describe the behavior of the intersection product with respect
to continuous deformation. Let $T$ denote the germ of a non-singular
curve at a point $t_0$, and consider a variety $V\subset M\times T$
which is flat over $T$. Then for $t\in T$ we have a well defined cycle
$V_t:=V\cap(M\times\{t\})$ (see~\cite[10.1]{fulton:it}).

If $\dim V=1$, then $V_t$ is a formal sum of points with positive
multiplicities. Conservation of numbers implies that the multiplicity
of the cycle $[p]$ in $V_{t_0}$ is given by the number of points in
$V_t$ (with multiplicities) converging to $p$ as $t\to t_0$.

To generalize this to arbitrary intersections, we have the following
continuity axiom~\cite[11.4.4.iii]{fulton:it}. Consider another
$W\subset M\times T$ which is flat over $T$, and suppose that $W_t$
meets $V_t$ properly for each $t\in T$. Then $V$ meets $W$ properly in
$M\times T$ and
\begin{equation} \label{eq:intersection-cont}
  (V\cdot W)_t = V_t \cdot W_t.
\end{equation}
We remark that in~\cite{fulton:it} this property is stated
axiomatically for the case where $W$ is a constant family
$W_0\times T$. To obtain the general case one considers the
intersection of $V\times_T W\subset M\times M\times T$ and the
diagonal $\Delta\times T\subset M\times M\times T$.

As a particular case of~\eqref{eq:intersection-cont}, when
$V\cdot W$ is a curve we obtain a description of the
multiplicity of $p$ in $V_{t_0}\cdot W_{t_0}$ as the number of points
of $(V\cdot W)_t$ converging to $p$ as $t\to t_0$.

If $V=\sum n_i [p_i]$ is an algebraic cycle of dimension $0$, then we
define $\deg V:=\sum n_i$. If $M$ is $\C^n$ or $(\C^*)^n$ and $V$ is
an algebraic cycle of pure dimension $k$, then we define $\deg V$ to
be $\deg V\cdot L$ where $L$ is a generic affine plane of codimension
$k$. By the continuity of intersection numbers, the degree function
is lower-semicontinuous on flat families.

\subsection{The Bernstein-Kushnirenko-Khovanskii Theorem}
\label{sec:bkt}

We give an overview of the notion of mixed volume and its relation to
the geometry of the torus group $(\C^*)^n$, encapsulated by the
Bernstein-Kushnirenko-Khovanskii (henceforth BKK) theorem. We follow
the presentation of~\cite{kk:convex-bodies}.

Recall that for $n$ convex bodies $\Delta_1,\ldots,\Delta_n$ in
$\R^n$, their \emph{mixed volume} is defined to be
\begin{equation}
  V(\Delta_1,\ldots,\Delta_n) = \pd{^n}{\lambda_1\cdots\partial \lambda_n} \vol(\lambda_1\Delta_1+\cdots+\lambda_n\Delta_n)
  \rest{\lambda_1=\cdots=\lambda_n=0^+}.
\end{equation}
The mixed volume is symmetric and multilinear, and generates
the volume function in the sense that $V(\Delta,\ldots,\Delta) = \vol(\Delta)$.
In fact, these properties completely determine the mixed volume
function.

Given a Laurent polynomial $P\in\C[x_1^{\pm1},\ldots,x_n^{\pm1}]$, we define its
support $\supp P\subset\Z^n$ to be the set of exponents appearing with non-zero
coefficients in $P$. For any set $A\subset\Z^n$ we denote by $\Delta_A$ the
convex hull of $A$ (in $\R^n$). Finally, we let $\Delta(P):=\Delta_{\supp P}$.

To each nonempty set $A\subset\Z^n$ we associate the vector space of polynomials
$P$ having $\supp P\subset A$.
\begin{Thm}[\protect{\cite{kushnirenko:bk,bernstein:bk}}]
  Let $A_1,\ldots,A_n\subset\Z^n$. Then for generic $P_i\in L_{A_i}$,
  the system of equations $P_1=\cdots=P_n=0$ admits exactly $\mu$ solutions
  in $(\C^*)^n$, where
  \begin{equation}
    \mu = n! V(\Delta_{A_1},\ldots,\Delta_{A_n}).
  \end{equation}
\end{Thm}

It follows by conservation of numbers that for any choice of
$P_i\in L_{A_i}$, not necessarily generic, the quantity $\mu$ above is
an upper bound for the number of \emph{isolated} solutions of
$P_1=\cdots=P_n=0$.

The following is a simple consequence for the computation of mixed
volumes.

\begin{Cor} \label{cor:mixed-vol-orth}
  Suppose that $\R^n=L_1\oplus L_2$ is an orthogonal decomposition, and
  that $\Delta_1,\ldots,\Delta_s\subset L_1$ and $\Delta_{s+1},\ldots,\Delta_n\subset L_2$
  are collections of convex bodies. Then
  \begin{equation} \label{eq:mixed-vol-orth}
    n! V(\Delta_1,\ldots,\Delta_n) = \big[ s! V(\Delta_1,\ldots,\Delta_s) \big]\cdot
    \big[ (n-s)! V(\Delta_{s+1},\ldots,\Delta_n) \big].
  \end{equation}
  In particular, if $s\neq\dim L_1$ then $V(\Delta_1,\ldots,\Delta_n)=0$.
\end{Cor}
\begin{proof}
  We may after an orthogonal change of coordinates assume that $L_1$
  is spanned by the $x_1,\ldots,x_s$ coordinates and $L_2$ is spanned
  by the $x_{s+1},\ldots,x_n$ coordinates. We will prove the claim
  under the assumption that $\Delta_1,\ldots,\Delta_n$ are Newton
  polytopes, i.e. convex hulls of subsets of the lattice $\Z^n$. The
  general claim can be established by continuous approximation
  (although in this paper we will only use the claim in this more
  restrictive sense).

  Under our assumption, the left hand side of~\eqref{eq:mixed-vol-orth} may
  be viewed as the number of zeros of a generic set of equations
  $P_1=\cdots=P_n=0$ with $P_i\in L_{\Delta_i}$. Since $P_1,\ldots,P_s$ and $P_{s+1},\ldots,P_n$
  involve disjoint sets of variables, it is clear that
  \begin{equation}
    \{P_1=\cdots=P_n=0\} = \{P_1=\cdots=P_s=0\} \times \{P_{s+1}=\cdots=P_n=0\}.
  \end{equation}
  The claim now follows by the BKK theorem.
\end{proof}

Let $\Delta_x:=\Delta(1+x_1+\cdots+x_n)$ denote the standard simplex
in the $x$-variables. For any convex body $\Delta$ and $j=0,\ldots,n$
we define the $j$-th (simplicial) \emph{quermassintegral} as
\begin{equation} \label{eq:qmi-def}
  \qmi_j(\Delta) = V(\underbrace{\Delta,\ldots,\Delta}_{n-j\text{ times}},
    \underbrace{\Delta_x,\ldots,\Delta_x}_{j\text{ times}}).
\end{equation}
We note that it is customary to use the Euclidean ball in place of the
standard simplex $\Delta_x$, but for our purposes the simplicial
normalization is more convenient.

Finally we remark on the affine case.
\begin{Rem}\label{rem:bk-affine}
  Recall that a \emph{co-ideal} in a semi-group is a set whose
  complement is an ideal. Suppose that
  $\Delta_1,\ldots,\Delta_n\subset\Z_{\ge0}^n$ are convex co-ideals
  and let $c_1,\ldots,c_n\in\C^n$. Then
  \begin{equation}
    \Delta(P_i)\subset\Delta_i \implies \Delta(P_i(x_1+c_1,\ldots,x_n+c_n))\subset\Delta_i
  \end{equation}
  In this case the BKK estimate holds even if one considers the number
  of solutions of a generic system of equations with assigned Newton
  polyhedrons in $\C^n$. Indeed, after a generic translation one may
  assume that all solutions lie in $(\C^*)^n$ and apply the usual BKK
  theorem.
\end{Rem}

\section{Milnor fibers and Multiplicities}
\label{sec:milnor-fibers}

In this section we recall a general notion of a Milnor fiber
of a deformation (due to L\^e), and its relation to the multiplicity
of an analytic function restricted to the trajectory of an
analytic vector field (due to Gabrielov).

In an effort to make the presentation self-contained, we have included
sketches for the proofs of most results that we shall use in the
sequel (with references for the original full proofs).

\subsection{L\^e's Milnor fiber of a deformation}
\label{sec:le-milnor}

We begin with the definition due to L\^e~\cite{le:milnor-fiber},
extending the notion of a Milnor fiber to a general deformation of an
analytic set. For simplicity we work in the ambient space
$M=\C^N\times\C$, and denote the projections to the second factor by
$e:M\to\C$.

\begin{Def} \label{def:milnor-fiber}
  Let $X\subset M$ and $x\in X$ with $e(x)=0$. Denote $X_\e:=X\cap e^{-1}(\e)$,
  and suppose that $e\rest X$ is flat in a neighborhood of $x$, i.e. $X_0$ is
  obtained as the limit of $X_\e$. We think of $X$ as a deformation
  of $X_0$.

  Then for any sufficiently small $\delta>0$ and $\delta\gg|\e|>0$, the
  homotopy type of the set $B_\delta(x)\cap X_\e$ is independent of
  the choice of $\delta,\e$ and is called the \emph{Milnor fiber} of
  $(X,x)$. Here $B_\delta(x)$ denotes the real Euclidean ball with
  respect to the standard metric.
\end{Def}

We remark that much of the material in this section can be generalized
to the case where $e$ is an arbitrary analytic function defined in a
neighborhood of $x$. The assumption that $e$ is a separate coordinate
slightly simplifies our presentation.

It is of importance for us that the Milnor fiber does not depend on
the coordinate system used to construct the balls. More generally, one
may also be interested in computing the Milnor fiber replacing the
balls $B_\delta(x)$ by a family of polydiscs. L\^e~\cite{le:monodromy}
defines the notion of \emph{privileged families} of neighborhoods for
this purpose.

The full definition is technical and goes beyond the
scope of this paper, but for the purpose of our exposition it will
suffice to specify one key property: if $\{P_\alpha\}$ is a privileged
family of neighborhoods, then for any sufficiently small $P_\alpha\subset P_\beta$
there exists $\e'>0$ such that for any $0<\e<\e'$ the inclusion
$P_\alpha\cap X_\e\hookrightarrow P_\beta\cap X_\e$ is a homotopy equivalence. L\^e shows
that the family of balls and the family of polydiscs in sufficiently
generic linear coordinates form privileged families.

With this definition we have the following standard fact.

\begin{Fact}\label{fact:priv-family}
  The Milnor fiber of $X$ at $x$ is the same (up to homotopy
  equivalence) when computed using any privileged family.
\end{Fact}
\begin{proof}
  Let $\{P_\alpha\},\{Q_\alpha\}$ be two privileged families around
  $x$. Choose sufficiently small
  $P_1\subset Q_1\subset P_2\subset Q_2$ from the two families. Then
  for any sufficiently small $\e\neq0$ the inclusions
 \begin{equation}
   \begin{split}
     P_1\cap X_\e\hookrightarrow P_2\cap X_\e \\
     Q_1\cap X_\e\hookrightarrow Q_2\cap X_\e
   \end{split}
  \end{equation}
  are homotopy equivalences. It follows that the inclusion map
  \begin{equation}
    Q_1\cap X_\e\hookrightarrow P_2\cap X_\e
  \end{equation}
  admits a right inverse $R$ and a left inverse $L$. Then
  $L\simeq L\cdot(\iota R)\simeq R$, so $\iota$ is a homotopy
  equivalence.
\end{proof}

We now present a result of L\^e~\cite{le:monodromy} explaining how the
Milnor fiber of a deformation is obtained from the Milnor fiber of a
hyperplane section by gluing cells corresponding to certain critical
points (cf. also~\cite[Proposition 2]{gabrielov:mult}). This may be
seen as a local complex analog of classical Morse theory. We present
only a special case which will suffice for our purposes.

\begin{Thm} \label{thm:le-gluing}
  Let $(X,x)$ be a deformation, and assume that in a
  neighborhood of $x$, the fibers $X_\e, \e\neq0$ are smooth. Let
  $\ell$ be an affine form satisfying $\ell(x)=0$ and denote
  $Y:=X\cap\ell^{-1}(0)$.

  Then, for $\ell$ sufficiently generic, we have
  \begin{enumerate}
  \item In a sufficiently small neighborhood of $x$, the fibers
    $Y_\e,\e\neq0$ are smooth.
  \item In a sufficiently small neighborhood of $x$, $\ell\rest{X_\e}$
    admits only isolated critical points for any $\e\neq0$.
  \item The Milnor fiber of $(X,x)$ is obtained from the Milnor
    fiber of $(Y,x)$ by attaching $\mu$ cells of dimension
    $\dim_\C X_\e$, where $\mu$ is the number of critical points of
    $\ell\rest{X_\e}$ converging to $x$ as $\e\to0$.
  \end{enumerate}
\end{Thm}

\begin{proof}
  The first two claims follow by a Bertini type argument which we omit
  (see~\cite[Lemma 2.2 and Remark 2.3]{le:milnor-fiber}). We continue with
  the proof of the third claim.
  
  We may assume that $\ell$ is just the first coordinate. Let
  $\Phi=(\ell,e):X\to\C^2$. According to~\cite[Theorem
  2.4]{le:milnor-fiber} for sufficiently generic $\ell$ there is a
  privileged family of polydiscs $P_\alpha=D_\alpha\times P'_\alpha$
  for $(X,x)$ such that
  \begin{equation}
    \begin{split}
      F=:P_\alpha\cap\Phi^{-1}(D_\alpha\times\{\e\}) &\text{ is the Milnor fiber of } (X,x)  \\
      F_0:=P_\alpha\cap\Phi^{-1}(\{0\}\times\{\e\}) &\text{ is the Milnor fiber of } (Y,x)
    \end{split}
  \end{equation}
  for sufficiently small $P_\alpha$ and (even smaller) $\e$. Now,
  since $X_\e$ is smooth by assumption, it remains only to interpolate
  these two spaces by
  \begin{equation}
    F_r = F \cap \abs{\ell}^{-1}([0,r])
  \end{equation}
  where $r$ runs from $0$ to the radius of $D_\alpha$. Since all
  critical points of $\ell$ are simple complex by assumption, all
  critical points of $\abs{\ell}$ are Morse of index $\dim_\C X_\e$.
  By classical Morse theory, whenever $r$ crosses the absolute value
  of a critical value of $\ell\rest F$, $F_r^+$ is glued with a cell
  of dimension $\dim_\C X_\e$. This concludes the proof.
\end{proof}

Theorem~\ref{thm:le-gluing} allows one to compute a cellular
decomposition for the Milnor fiber using induction on dimension. It
motivates the following definitions. Let $(X,x)$ be a deformation.
Assume that in a neighborhood of $x$, the fiber $X_\e,\e\neq0$ is
smooth. Fix generic functionals
\begin{equation} \label{eq:vell-def}
  \vell=(\ell_1,\ldots,\ell_{\dim_\C X_0}), \qquad \ell_i\in M^*
\end{equation}
For $k=1,\ldots,\dim_\C X_0$ we denote by $\las(x)$ the affine
space
\begin{equation} \label{eq:las-def}
  \las(x) := \ell_1^{-1}(\ell_1(x))\cap\cdots\cap\ell_k^{-1}(\ell_k(x)).
\end{equation}
We note that under the generic assumption that $e$ does not belong to
the span of $\vell$, the maps $e:\las(x)\to\C$ are flat. As usual
we denote by $\las(x)_\e:=\las(x)\cap e^{-1}(\e)$ the
corresponding fibers.

\begin{Def}\label{def:polar}
  For $k=1,\ldots,\dim_\C X_0+1$, we define the $k$-th \emph{polar
    variety} of $X$, denoted $\pv(X)$, as follows
  \begin{equation}
    \pv(X) = \clo\left[ \{\d\ell_1\wedge\cdots\d\ell_k\wedge\d e\rest X=0\}\setminus\{e=0\}  \right]
  \end{equation}
  where $\clo$ denotes analytic closure. In a neighborhood of $x$
  where all fibers $X_\e,\e\neq0$ are smooth, this may be stated as
  follows: we define $\pv(X)$ as the locus where
  $\d\ell_1,\ldots,\d\ell_k$ are linearly dependent on the tangent
  space of the fiber $X_\e$ for $\e\neq0$, and complete this to a flat
  family over $e=0$.

  For any $\e$, we denote by $\pv(X)_\e$ the analytic cycle
  in $\C^N\times\{\e\}$ as defined in section~\secref{sec:cycles}.
\end{Def}

Recall that an analytic stratification $\{Z_\alpha\}$ of an analytic
subset of $X_0$ is said to satisfy \emph{Thom's $A_e$} condition if
the following condition holds: for any sequence of points $x_i\in X$
converging to a point $x'\in Z_\alpha$, if the sequence of tangent
spaces $T_{x_i}X_{e(x)}$ converges to a limit $T$, then
$T_{x'}Z_\alpha\subset T$. Such stratifications always exist by a
result of Hironaka~\cite{hironaka:startifications}.

The following two propositions establish the key properties of
$\las(x)$ and $\pv(X)$, and their relation to the cellular structure
of the Milnor fiber. Polar varieties have been used for such purposes
extensively in the literature
(see~\cite{le:polar-curves,massey:numerical} for a survey), and the
ideas for the proofs are standard. However, we are not aware of a
suitable reference in this generality, and we therefore present full
proofs.

\begin{Prop} \label{prop:pv-proper}
  Let $(X,x)$ be a deformation, and assume that in a neighborhood of
  $x$, the fibers $X_\e, \e\neq0$ are smooth. There exists a
  neighborhood of $x$ and generic $\vell$ such that for every
  $k=1,\ldots,\dim_\C X_0+1$ and sufficiently small $\e$ (including
  zero) we have:
  \begin{enumerate}
  \item $\pv(X)_\e$ has pure dimension $k-1$
  \item $\pv(X)_\e$ intersects $\las[k-1](x)_\e$ properly (i.e.
    at isolated points).
  \end{enumerate}
\end{Prop}
\begin{proof}
  We describe the choice of generic $\vell$. Let $\{Z^1_\alpha\}$ be
  a stratification of $X_0$ satisfying Thom's $A_e$ condition. We may
  restrict to an open neighborhood of $x$ where the only zero-dimensional
  strata (if any) is $\{x\}$. Let $\ell_1$ be transversal to all the
  other strata in $\{Z^1_\alpha\}$.

  Consider now a stratification $\{Z^2_\alpha\}$ of $X_0\cap \las[1](x)_0$
  refining $\{Z^1_\alpha\}$. Once again, we may restrict to an open neighborhood
  of $x$ where the only zero-dimensional strata (if any) is $\{x\}$. Let $\ell_2$
  be transversal to all the other strata in $\{Z^2_\alpha\}$.

  Continuing in this fashion we obtain for $k=1,\ldots,\dim_\C X_0+1$
  stratifications $\{Z^k_\alpha\}$ of $X_0\cap \las[k-1](x)_0$, and
  functional $\ell_k$ transversal to all strata in $\{Z^k_\alpha\}$
  except perhaps $\{x\}$. Note that these stratifications all satisfy
  Thom's $A_e$ condition (since the condition is preserved under
  refinement). Denote by $U$ the open neighborhood of $x$ in which
  all the stratifications were constructed.

  We now proceed with the proof. First note that for $\e\neq0$,
  $\pv(X)_\e$ is a determinantal variety given by
  \begin{equation}
    \pv(X)_\e = \d\ell_1\wedge\cdots\wedge\d\ell_k \rest{X_\e}=0
  \end{equation}
  and as such, each of its components has dimension at least $k-1$.
  Since $\pv(X)_0$ is obtained by a flat limit, the same is
  true for it.

  The other direction of (1), as well as (2), will be proved by
  continuity once we show that $\pv(X)_0$ intersects
  $\las[k-1](x)_0$ properly. More specifically, we will show that in $U$
  this intersection contains only $x$.

  Assume to the contrary that
  \begin{equation}
    x' \in \pv(X)_0 \cap \las[k-1](x)_0 \cap U, \qquad x'\neq x.
  \end{equation}
  Since $\pv(X)_0$ is defined by a flat limit, there exists
  a sequence $x_i\in\pv(X)_{e(x_i)}$ with $e(x_i)\neq0$
  and $x_i\to x'$. By compactness of the Grassmannian we may assume
  that $T_{x_i}X_{e(x_i)}$ converges to a limit $T$.

  For $p=1,\ldots,k$ denote by $Z^p_0$ the strata in $\{Z^k_\alpha\}$
  containing $x'$. By construction there exists a vector $v_p\in T_{x'}Z^p_0$
  such that $\d\ell_p(v_p)=1$. Since $Z^p_0$ is an analytic subset of
  $X_0\cap \las[p-1](x)_0$, we have $\d\ell_q(v_p)=0$ for $q<p$. It follows that
  the matrix $(\d\ell_p(v_q))_{p,q=1\ldots k}$ is upper triangular with determinant 1.

  By Thom's $A_e$ condition for each of the stratifications
  $\{Z^k_\alpha\}$, we have that $v_1,\ldots,v_k\in T$. Thus
  eventually $\d\ell_1,\ldots,\d\ell_k$ become linearly independent on
  $T_{x_i}X_{e(x_i)}$, contradicting our assumption that
  $x_i\in\pv(X)_{e(x_i)}$.
\end{proof}

Finally, we present a proposition expressing the cellular structure
of a Milnor fiber in terms of the polar varieties $\pv(X)$
and their intersections with $\las(x)$.

\begin{Prop}\label{prop:polar-cells}
  Let $(X,x)$ be a deformation, and assume that in a neighborhood of
  $x$, the fibers $X_\e, \e\neq0$ are smooth of dimension
  $d:=\dim_\C X_\e$. For any sufficiently generic $\vell$, the Milnor
  fiber of $(X,x)$ admits a cellular structure where the number of
  $k$-cells, denoted $c_q$, is given by
  \begin{equation}\label{eq:polar-cells-formula}
    c_{d+1-k} = i(x;\pv(X)_0\cdot \las[k-1](x)_0; \C^N).
  \end{equation}
  Thus, the $q$-th Betti number of the Milnor fiber is bounded
  by $c_q$. In particular the $q$-th Betti number vanishes for $q>d$.
\end{Prop}
\begin{proof}
  Applying Theorem~\ref{thm:le-gluing} inductively, we obtain a
  cellular decomposition for the Milnor fiber where $c_{d+1-k}$ is
  equal to the number of critical points of
  $\ell_k\rest{X_\e\cap \las[k-1](x)}$ converging to $x$ as as
  $\e\to0$, i.e. to the number of points in
  $\pv(X)_\e\cdot \las[k-1](x)_\e$ converging to $x$ as $\e\to0$.
  Since the intersection $\pv(X)_0\cdot \las[k-1](x)_0$ is proper at
  $x$ according to Proposition~\ref{prop:pv-proper}, we
  obtain~\eqref{eq:polar-cells-formula} by the continuity of
  intersection numbers.
\end{proof}

\subsection{Milnor fibers and multiplicities}
\label{sec:milnor-mult}

In this section we present the main ideas of~\cite{gabrielov:mult},
relating the multiplicity of an analytic function restricted to the
trajectory of an analytic vector field to the Euler characteristics of
the Milnor fibers of certain deformations.

Let $p\in\C^n$. Let $V$ be an analytic vector field and $P$ an
analytic function, both defined in a neighborhood of $p$. As
in~\secref{sec:le-milnor}, we consider the ambient space
$M=\C^n\times\C$. Finally consider an analytic deformation $P_e$ of
$P$, i.e.
\begin{equation}
  P_e(x_1,\ldots,x_n) \in \cO_{(0,p)}(e,x_1,\ldots,x_n) \qquad P_0(x)\equiv P(x).
\end{equation}
Following Gabrielov, we define for every $r\in\N$
\begin{equation}
  X^r = \clo\left[ \{ P_e = VP_e = \cdots = V^{r-1} P_e = 0 \} \setminus \{ e=0 \} \right]
\end{equation}
where $\clo$ denotes analytic closure. In other words, we define
$X^r$ by the vanishing of the first $r$ derivatives outside $e=0$
and complete this variety as a flat family over $e=0$. We denote
by $F_p^r$ the Milnor fiber of $(X^r,p)$.

Gabrielov's key insight is Theorem~\ref{thm:gab-main}, expressing the
multiplicity $\mult_p^V P$ in terms of the Euler characteristics of
the Milnor fibers defined above. The proof we present is based on an
idea of Khovanskii, and appeared in~\cite{gabrielov:mult}.

We remark that the presentation of this proof in~\cite{gabrielov:mult}
contains a small gap. A second, complete proof is also given
in~\cite{gabrielov:mult}. However, for the multi-dimensional
generalization developed in~\cite{gk:mult} it is necessary to use the
former proof. We therefore pause to present a lemma making the
argument precise.

The following lemma was suggested to the author by David B. Massey.
In~\cite{ks:sheaves}, this lemma is presented in the general context
of constructible sheaves, and is stated as a result concerning
cohomology. We give below a simplified formulation in our more
elementary context, and a proof sketch (adapted
from~\cite{ks:sheaves}) in the homotopic category.

\begin{Lem}[\protect{\cite[Lemma 8.4.7]{ks:sheaves}}] \label{lem:milnor-phi}
  In the notations of Definition~\ref{def:milnor-fiber}, let $\phi:X\to\R_{\ge0}$
  be a nonnegative real analytic function. Suppose that $\phi$ is proper in a
  neighborhood of $x$ and that and $\phi^{-1}(0)=\{x\}$.

  Then the Milnor fiber of $(X,x)$ is homotopy equivalent to
  $\phi^{-1}(B_\delta)\cap X_\e$ for sufficiently small
  $0\neq\e\ll\delta\ll1$.
\end{Lem}

With $\phi$ given by the squared-distance to $x$ one obtains the usual
expression for the Milnor fiber.

\begin{proof}[Proof sketch]
  Fix a stratification of $X$ refining a stratification of
  $X_0$, which satisfies the Whitney B condition as well as
  Thom's $A_e$ condition.

  By an argument of Bertini-Sard type~\cite[Lemma 8.4.7]{ks:sheaves},
  one checks that $\phi$ has a discrete set of stratified critical
  values. In particular, for sufficiently small $\delta<\delta'$,
  $\phi$ has no stratified critical values between $\delta$ and
  $\delta'$. Moreover, using Thom's $A_e$ condition one proves that
  for sufficiently small $\e$, the restriction of $\phi$ to $X_\e$
  also has no critical values between $\delta$ and $\delta'$.

  It follows from the stratified Morse lemma that there is a homotopy
  equivalence between $\phi^{-1}(B_\delta) \cap X_\e$ and
  $\phi^{-1}(B_{\delta'}) \cap X_\e$. One can complete the proof as
  the proof of Fact~\ref{fact:priv-family}.
\end{proof}

We are now ready to state Gabrielov's main theorem.

\begin{Thm}[\protect{\cite[Theorem 1]{gabrielov:mult}}]
  \label{thm:gab-main}
  Let $p\in\C^n$ be a non-singular point of $V$ and suppose that
  $\mult_p^V P$ is finite. Then
  \begin{equation} \label{eq:mult-sum-chi}
    \mult_p^V P = \sum_{r=1}^\infty \chi(F^r_p)
  \end{equation}
\end{Thm}
\begin{proof}[Proof suggested by Khovanskii.]
  We first note that since $\mu:=\mult_p^V P$ is finite, it follows
  that $F^r_p=\emptyset$ for $r>\mu$, and hence the
  sum~\eqref{eq:mult-sum-chi} is finite.
  
  We may choose local analytic coordinates $(z_1,z')$ around $p$ such
  that $V=\pd{}{z_1}$. Denote by $\pi:\C^n\to\C^{n-1}$ the projection
  to the $z'$ coordinates. By assumption, $\pi\rest{X^0_0}$ has
  ramification of multiplicity $\mu:=\mult_p^V P$ at the origin. Then
  one can choose a neighborhood of the form
  $U=D\times B$ of the origin such that the
  $\pi\rest{U\cap X^0_0}$-fiber of any point in $B$ has exactly
  $\mu$ points counted with multiplicities. It also follows that the
  same is true for $\pi\rest{U\cap X^0_\e}$ for sufficiently small
  $\e$. We would essentially like to think of $U\cap X^r_\e$ as representing
  the Milnor fiber $F^r_p$. This requires a small technical justification,
  as follows.

  Let $\phi(z_1,z',e)=\norm{z'}^2+\norm{e}^2$. Then applying Lemma~\ref{lem:milnor-phi}
  for each $U\cap X^r$ and $\phi$, we have
  \begin{equation}
    F^r_p \simeq \phi^{-1}(\delta')\cap X^r_\e = X^r_\e\cap (D\times B_\delta) =: Y^r,
    \quad \text{where } \delta=\sqrt{\delta'-\abs{\e}^2}
  \end{equation}
  for sufficiently small $\e\ll\delta'\ll1$. Fix such a pair.

  We know that the $\pi\rest{Y^0}$-fiber of any point $z'\in B_\delta$ contains
  $\mu$ points, counted with multiplicities. Since the points of $Y^r$
  are exactly the points where $\pi$ has multiplicity $k+1$, it follows
  by a Riemann-Hurwitz type counting argument that for any $z'\in B_\delta$,
  \begin{equation}
    \mu = \sum_{k=1}^\infty \#\{\left[\pi\rest{Y^r}\right]^{-1}(z')\}.
  \end{equation}
  Using the Fubini theorem for integration over Euler
  characteristic~\cite{viro:euler-calculus}, we obtain
  \begin{equation}
    \begin{split}
      \mu = \int_{B_\delta} \mu \d\chi &= \int_{B_\delta} \sum_{r=1}^\infty \#\{\left[\pi\rest{Y^r}\right]^{-1}(z')\} \d\chi(z') \\
       &=\sum_{r=1}^\infty \int_{B_\delta} \chi \left(\left[\pi\rest{Y^r}\right]^{-1}(z')\right) \d\chi(z') \\
       &=\sum_{r=1}^\infty \chi(Y^r)  =\sum_{r=0}^\infty \chi(F^r_p).
    \end{split}
  \end{equation}
\end{proof}

The usefulness of Theorem~\ref{thm:gab-main} becomes apparent in view
of the following lemma, which guarantees the existence of sufficiently
generic deformations. We omit the proof, which is a standard exercise
in Sard type arguments, and refer the reader to~\cite{gabrielov:mult}
for details. We say that $X^r_\e$ is \emph{effectively smooth} at a
point $x$ if it is smooth, and moreover
$\d P_\e\wedge\cdots\wedge\d(V^{r-1}P_\e)\rest{X^r_\e}$ is non-zero at
$x$.

\begin{Lem}[\protect{\cite[Lemma 1]{gabrielov:mult}}]\label{lem:gab-sard}
  Let $\ell$ be a germ of an analytic function and suppose that
  $V\ell(p)\neq0$. Let $P^c=c_0+\cdots+c_n\ell^{n-1}$ where the
  coefficients $c_i$ are chosen generically, and consider the
  deformation $P^c_e(x)=P(x)+eP^c(x)$. Then there exists a neighborhood
  $U$ of $p$ such that for any sufficiently small $0\neq\e\ll1$,
  $X^r_\e$ is an effectively smooth $n-r$ dimensional set in $U$. In
  particular, $X^r_\e$ is empty in $U$ for $r>n$.
\end{Lem}

\section{An estimate for the Betti numbers of the Milnor fiber}
\label{sec:betti}

In this section we present an estimate for the Betti numbers of
the Milnor fiber of a deformation, under a smoothness assumption.
The estimate is expressed in terms of the geometry of the polar
varieties.

In this section we consider the ambient manifold $M=(\C^*)^n\times\C$
with coordinate ring $R=\C[x_1^{\pm1},\ldots,x_n^{\pm1},e]$ or
$M=\C^n\times\C$ with coordinate ring $R=\C[x_1,\ldots,x_n,e]$. Let
$P^1_e,\ldots,P^r_e\in R$, where as usual we think of $e$ as a
deformation parameter. Let $\Delta\subset\Z^n$ denote the convex hull
of $\Delta(P^1_e),\ldots,\Delta(P^r_e)$ for a fixed generic value of $e$.

We denote
\begin{equation}
  X^r = \clo\left[ \{ P^1_e = P^2_e = \cdots = P^r_e = 0 \} \setminus \{ e=0 \} \right]
\end{equation}
where $\clo$ denotes analytic closure. In other words, we define
$X^r$ by the vanishing of $P^1_e,\ldots,P^r_e$ outside $e=0$
and complete this variety as a flat family over $e=0$. We denote
by $F_p$ the Milnor fiber of $(X,p)$.

Let $\Sigma(X^r)$ denote the set of points where the fiber $X_e^r$ is not
effectively smooth (completed as a flat family over $e=0$), i.e.
\begin{equation}\label{eq:sigma-def}
  \Sigma(X^r) = \clo\left[\{ \d P^1_e\wedge\cdots\wedge\d P^r_e\wedge\d e=0 \} \setminus\{e=0\} \right]
\end{equation}
We will say that a point $p\in X^r_0$ is \emph{good} if
$p\not\in\Sigma(X^r)$. Our goal is to estimate the Betti numbers
$b_i(F^r_p)$ at good points $p$ in terms of the Newton polytope
$\Delta$. More specifically, we give an appropriate definition for
globally defined polar varieties, whose degrees are bounded in terms
of the Newton polytopes, and show that the geometry of these polar
varieties controls the Betti numbers.

\subsection{The polar varieties}

In this section we keep the notation of~\secref{sec:le-milnor}.
Naturally, our objective is to obtain upper bounds on the Betti
numbers of the Milnor fiber in terms of the polar varieties through
Proposition~\ref{prop:polar-cells}. One could attempt to use the
definition of polar varieties given in Definition~\ref{def:polar}
directly. However, in the global context this definition gives rise to
certain degeneracies (for instance, where the sets $X^r_e$ are
singular, and when the functionals $\vell$ are not sufficiently
generic), making the polar varieties more difficult to study. Since we
are interested primarily in the behavior of these varieties around
good points, we opt to use a refined definition which agrees with
Definition~\ref{def:polar} in a neighborhood of such points while
eliminating some of the more complicated degenerate behavior.

\begin{Def} \label{eq:bpv-def}
  For $k=1,\ldots,n-r+1$ and $\vell$ as in~\eqref{eq:vell-def}, we
  define the set $\tpv(X^r)$ by
  \begin{equation}
    \tpv(X^r) := C_k \left[ \pv(X^r) \setminus (\Sigma(X^r)\cup\pv[k-1](X^r)) \right]
  \end{equation}
  where $C_k(A)$ denotes the union of the $k$-dimensional components
  of $A$ which are not contained in a fiber $e=\const$. We define the
  \emph{refined polar variety}, denoted $\bpv(X^r)$, to be the Zariski
  closure of $\tpv(X^r)$.
\end{Def}

By definition, $\bpv(X^r)$ is a $k$-dimensional flat family, with pure
$k-1$-dimensional fibers $\bpv(X^r)_\e$. Let $p\in X_0$ be a good
point and suppose that $\vell$ is sufficiently generic. Then by
Proposition~\ref{prop:pv-proper}, in a neighborhood of $p$ the set
$\pv(X^r)$ has pure dimension $k$ and the set $\pv[k-1](X^r)$ has pure
dimension $k-1$. It follows that in a neighborhood of $p$,
$\bpv(X^r)=\pv(X^r)$.

We now consider the degree of the refined polar variety. We
begin with the case of the torus $M_0=(\C^*)^n$.

\begin{Prop}\label{prop:pv-deg}
  For any $\e\in\C$, we have the bound
  \begin{equation}
    \deg \bpv(X^r)_\e \le \tbinom{n}{r+k-1}n! \qmi_{k-1}(\Delta+\Delta_x)
  \end{equation}
  where $\deg$ denotes degree in $(\C^*)^n$, $\qmi$ denotes the
  quermassintegral, and $\Delta_x$ denotes the standard simplex in the
  $x$ variables.
\end{Prop}
\begin{proof}
  We may assume without loss of generality that $\Delta$ contains the
  origin. Indeed, one can always translate $\Delta$ to achieve this by
  multiplying the equations $P^1,\ldots,P^r$ by a common monomial.
  This does not affect the set $X^r$ or $\Sigma(X^r)$ outside of the
  coordinate axes (which lie outside $(\C^*)^n$) and thus it is
  straightforward to check that it does not affect the refined polar
  variety $\bpv(X^r)$.
  
  By the lower semicontinuity of the degree function in flat families,
  it suffices to prove the claim for a generic fiber. Let $L$ be a
  generic affine plane of codimension $k-1$ in $(\C^*)^n$. Since the
  generic fiber $\bpv(X^r)_\e$ has pure dimension $k-1$ and the
  generic fiber $\bpv(X^r)_\e\setminus\tpv(X^r)_\e$ has strictly
  smaller dimension, we may assume that $L$ intersects $\bpv(X^r)_\e$
  only in points of $\tpv(X^r)_\e$. Let $f_1,\ldots,f_{k-1}$ denote
  $k-1$ affine linear functionals defining $L$.

  We now restrict attention to a particular generic fiber
  $e^{-1}(\e)\simeq(\C^*)^n$. All exterior derivatives computed below
  are taken with respect to this ambient space. At any point
  $p\in\tpv(X^r)_\e$, the differentials
  \begin{equation}
    \d P^1_\e,\ldots,\d P^r_\e, \d\ell_1,\ldots,\d\ell_{k-1}
  \end{equation}
  are linearly independent, while the differentials
  \begin{equation}
    \d P^1_\e,\ldots,\d P^r_\e, \d\ell_1,\ldots,\d\ell_{k-1},\d\ell_k
  \end{equation}
  are linearly dependent. Thus there exists one and only one linear
  dependence of the form
  \begin{equation}
    \d\ell_k = \l_1\d P^1_\e+\cdots+\l_r P^r_\e+\l_{r+1}\d\ell_1+\cdots+\l_{r+k-1}\d\ell_{k-1}
  \end{equation}
  with $\l_1,\ldots,\l_{r+k-1}\in\C$ at $p$.

  In other words, each intersection between $\tpv(X^r)_\e$ and $L$
  corresponds to an isolated solution of the following system of
  equations
  \begin{equation}\label{eq:bpv-deg-system}
    \begin{aligned}
      &P^j_\e=0 \qquad &j=1,\ldots,r \\
      &\pd{\ell_k}{x_j} = \sum_{i=1}^r \l_i \pd{P^i_\e}{x_j}+\sum_{i=1}^{k-1} \l_{r+i}\pd{\ell_i}{x_j}=0
      \qquad &j=1,\ldots,n \\
      &f_j = 0 \qquad &j=1,\ldots,k-1
    \end{aligned}
  \end{equation}

  Denote by $\Delta_x$ (resp. $\Delta_\l$) the standard simplex in
  the $x$ (resp. $\l$) variables. Then the system above has Newton
  polytopes bounded by
  \begin{equation}
    \underbrace{\Delta}_{r\text{ times}},
    \underbrace{\Delta-\Delta_x+\Delta_\l}_{n\text{ times}},
    \underbrace{\Delta_x}_{k-1\text{ times}}
  \end{equation}
  We now estimate the number of solutions of~\eqref{eq:bpv-deg-system}
  by the BKK theorem. Since the Newton polytopes above are invariant
  under translation in the $\l$ variables, we have that the number of
  solutions in $(\C^*)^n\times\C^{r+k-1}$ is bounded by the mixed
  volume
  \begin{equation}
    (n+r+k-1)! V(\underbrace{\Delta}_{r\text{ times}},
    \underbrace{\Delta-\Delta_x+\Delta_\l}_{n\text{ times}},
    \underbrace{\Delta_x}_{k-1\text{ times}}).
  \end{equation}
  We expand this mixed volume by linearity. In the expansion, if the
  $\Delta_\l$ term is not taken $r+k-1$ times out of the $n$ appearances,
  then the mixed volume vanishes by Corollary~\ref{cor:mixed-vol-orth}.
  Thus, again by Corollary~\ref{cor:mixed-vol-orth} the mixed volume
  is equal to
  \begin{equation}
    \tbinom{n}{r+k-1} n! V(\underbrace{\Delta}_{r\text{ times}},
    \underbrace{\Delta-\Delta_x}_{n-r-k+1\text{ times}},
    \underbrace{\Delta_x}_{k-1\text{ times}})
  \end{equation}
  and since the mixed volume is invariant under translation and
  monotone with respect to each argument, we finally obtain that the
  number of solutions of~\eqref{eq:bpv-deg-system} is bounded by
  \begin{equation}
    \tbinom{n}{r+k-1}n! V(\underbrace{\Delta+\Delta_x}_{n-k+1\text{ times}},
    \underbrace{\Delta_x}_{k-1\text{ times}})
  \end{equation}
  as stated.
\end{proof}

We move now to the case of the affine space $M_0=\C^n$. Suppose that
$\Delta\subset\Z_{\ge0}^n$ is a convex co-ideal. Then
\begin{equation}
  \Delta(P)\subset\Delta \implies \Delta(P_{x_i})\subset\Delta \qquad i=1,\ldots,n
\end{equation}
In this case one can repeat the proof of Proposition~\ref{prop:pv-deg},
in combination with Remark~\ref{rem:bk-affine} to obtain the following.

\begin{Prop}\label{prop:pv-deg-affine}
  Suppose that $\Delta\subset\Z_{\ge0}^n$ is a convex co-ideal. Then
  for any $\e\in\C$ we have the bound
  \begin{equation}
    \deg \bpv(X^r)_\e \le \tbinom{n}{r+k-1}n! \qmi_{k-1}(\Delta)
  \end{equation}
  where $\deg$ denotes degree in $\C^n$ and $\qmi$ denotes the
  quermassintegral.

  In particular, if $P^1,\ldots,P^r$ are polynomials (with respect
  to $x$) of degrees bounded by $d$, then
  \begin{equation}
    \deg \bpv(X^r)_\e \le \tbinom{n}{r+k-1} d^{n-k+1}.
  \end{equation}
\end{Prop}

\subsection{Upper bounds for Betti numbers}

In this subsection we present two upper bounds for the Betti
numbers of the Milnor fiber. The first of these is given in
terms of intersection numbers between the polar varieties
$\bpv(X^r)$ and the affine spaces $\las[k-1](p)$.

\begin{Thm}\label{thm:fiber-good-point-bd}
  Let $S\subset M_0$ be a finite collection of good points (i.e.
  points $p$ such that $X^r_\e$ is effectively smooth in a
  neighborhood of $p$ for $\e\neq0$). Fix $\vell$ sufficiently
  generic.

  Then for $k=0,\ldots,n-r$ and for any good $p\in S$ , we have
  \begin{equation} \label{eq:betti-good-point-bound}
    b_k(F^r_p) \le i(p; \bpv[n-r-k+1](X^r)_0 \cdot \las[n-r-k](p)_0;M_0)
  \end{equation}
\end{Thm}
\begin{proof}
  The statement follows by application of Proposition~\ref{prop:polar-cells},
  after noting that $\bpv[n-r-k+1](X^r)$ agrees with the polar variety $\pv[n-r-k+1](X^r)$
  in a neighborhood of any good point $p$.
\end{proof}

Next, we give a bound that holds uniformly at all good points $p\in M_0$.

\begin{Thm}\label{thm:fiber-uniform-bd}
  Fix $\vell$ sufficiently generic. Then for $k=0,\ldots,n-r$ and for
  any good $p\in M_0$ , we have
  \begin{equation} \label{eq:betti-uniform-bound}
    b_k(F^r_p) \le \degf_{\bpv[n-r-k+1](X^r)_0}(p)
  \end{equation}
\end{Thm}
\begin{proof}
  The proof follows by application of the results from the appendix.
  Namely, consider function
  \begin{equation}
    f:M_0\to\N \qquad f(p)=\begin{cases}
      b_k(F^r_p) & \text{if $p$ is good} \\
      0 & \text{otherwise}
    \end{cases}
  \end{equation}
  By Corollary~\ref{cor:degf-complexity} the function
  $F_\vell := \degf_{\bpv[n-r-k+1](X^r)_0}$ has uniformly bounded
  complexity independent of $\vell$. By
  Proposition~\ref{prop:sc-compact} there exists a finite set
  $S\subset M_0$ such that for any $\vell$,
  $f\rest S\le F_\vell\rest S$ implies $f\le F_\vell$.

  Choose $\vell$ sufficiently generic so that Theorem~\ref{thm:fiber-good-point-bd}
  applies for that set $S$. We claim that $f\rest S\le F_\vell\rest S$. Indeed,
  the inequality is trivial for points $p\in S$ which are not good, and
  for good points it follows from~\eqref{eq:betti-good-point-bound} and the
  simple observation
  \begin{equation}
    i(p; \bpv[n-r-k+1](X^r)_0 \cdot \las[n-r-k](p)_0;M_0) \le
    \degf_{\bpv[n-r-k+1](X^r)_0}(p)
  \end{equation}
\end{proof}

\section{Multiplicity estimates}
\label{sec:mult-estimates}

In this section we turn to the subject of multiplicity estimates. Once
again we consider the ambient manifold $M=(\C^*)^n\times\C$ with
coordinate ring $R=\C[x_1^{\pm1},\ldots,x_n^{\pm1},e]$ or
$M=\C^n\times\C$ with coordinate ring $R=\C[x_1,\ldots,x_n,e]$, where
$e$ is viewed as the parameter of a deformation.

Consider a Laurent vector field $V$ and a Laurent polynomial $P$,
\begin{equation} \label{eq:VP-torus}
  \begin{split}
    V = \sum_{i=1}^n Q_i \pd{}{x_i} \qquad & Q_i\in\C[x_1^{\pm1},\ldots,x_n^{\pm1}], \\
                                          & P\in\C[x_1^{\pm1},\ldots,x_n^{\pm1}].
  \end{split}
\end{equation}
Denote by $\Delta(P)$ the Newton polytope of $P$, and by $\Delta(V)$
the Newton polytope of $V$, where to each monomial
$x^\alpha\pd{}{x_i}$ we associate the exponent of the monomial
$x^\alpha/x_i$ in $\Z^n$.

\subsection{The multiplicity cycles}

Recall the notations of~\secref{sec:milnor-mult}. Let $\l=(\vell,c)$
where $c$ denotes the parameters defining the deformation $P^c$ of $P$
given in Lemma~\ref{lem:gab-sard}. We will denote $P^c$ by $P^\l$ to
simplify the notation.

We start by defining a collection of algebraic cycles which play a key
role in our multiplicity estimates.

\begin{Def}\label{def:mc}
  For $k=0,\ldots,n-1$ and $\vell$ as in~\eqref{eq:vell-def}, we
  define the $k$-th \emph{multiplicity cycle} of the deformation
  $P^\l$, denoted $\mc^k(P^\l)$, to be the $k$-cycle in $M_0$
  given by
  \begin{equation}
    \mc^k(P^\l) := \sum_{r=1}^{n-k} \bpv[k+1](X^r)_0
  \end{equation}
\end{Def}

The motivation for this definition becomes apparent in light of the
following theorem, describing the behavior of the multiplicity function
$\mult_p^V P$ in terms of the multiplicity cycles.

\begin{Thm}\label{thm:mult-mc-bound}
  Let $S\subset M_0$ be a finite collection of points, and assume that
  for every $p\in S$ the vector field $V$ is non-singular and
  $\mult_p^V P<\infty$. Fix $\l$ sufficiently generic with respect to
  $S$.

  Then for any $p\in S$ we have
  \begin{equation} \label{eq:mult-mc-bound}
    \mult_p^V P \le \sum_{k=0}^{n-1} i(p; \mc^k(P^\l)\cdot\las(p);M_0) 
  \end{equation}
\end{Thm}
\begin{proof}
  By Lemma~\ref{lem:gab-sard} we may assume that each $p\in S$ is
  a good point of the corresponding deformation $P^\l$, and $F^r_p$ is empty
  for $r>n$. Therefore, for any $p\in S$ we have
  \begin{multline}
    \mult_p^V P \ot=i \sum_{r=1}^\infty \chi(F^r_p) \ot={ii} \sum_{r=1}^n \chi(F^r_p) 
    \ot\le{iii} \sum_{r=1}^n \sum_{k=0}^{n-r} b_k(F^r_p) \\
    \ot\le{iv} \sum_{r=1}^n \sum_{k=0}^{n-r} i(p; \bpv[n-r-k+1](X^r)_0 \cdot \las[n-r-k](p)_0;M_0) \\
    \ot\le{v} \sum_{k=0}^{n-1} i(p; \mc^k(P^\l)\cdot\las(p);M_0)
  \end{multline}
  where (i) follows from Theorem~\ref{thm:gab-main}; (ii) follows
  since $F^r_p$ is empty for $r>n$; (iii) follows since the Euler
  characteristic is bounded by the sum of the Betti numbers, and the
  $k$-th Betti number of $F^r_p$ vanishes for $k>n-r$; (iv) follows
  from Theorem~\ref{thm:fiber-uniform-bd}; and (v) is a
  re-summation.
\end{proof}

Next, we give a bound that holds uniformly at all points $p\in M_0$
where $V$ is non-singular and $\mult_p^V P<\infty$.

\begin{Thm}\label{thm:mult-mc-uniform-bound}
  There exists $\l$ such that the following holds (in fact, for any
  sufficiently generic choice of $\l$): for every point $p\in M_0$
  where $V$ is non-singular and $\mult_p^V P<\infty$ we have
  \begin{equation} \label{eq:mult-mc-uniform-bound}
    \mult_p^V P \le \sum_{k=0}^{n-1} \degf_{\mc^k(P^\l)} (p).
  \end{equation}
\end{Thm}
\begin{proof}
  To show that for a sufficiently generic $\l$ the
  bound~\eqref{eq:mult-mc-uniform-bound} holds uniformly over the
  points $p\in M_0$ where $V$ is non-singular and the multiplicity is
  finite, we proceed as in the proof of
  Theorem~\ref{thm:fiber-uniform-bd}. Consider the function
  \begin{equation}
    f:M_0\to\N \qquad f(p)=\begin{cases}
      \mult_p^V P & \text{if $V(P)\neq0$ and $\mult_p^V P<\infty$} \\
      0 & \text{otherwise}
    \end{cases}
  \end{equation}

  By Corollary~\ref{cor:degf-complexity} the function
  $F_\vell := \degf_{\mc^k(P^\l)}$ has uniformly bounded
  complexity independent of $\l$. By
  Proposition~\ref{prop:sc-compact} there exists a finite set
  $S\subset M_0$ such that for any $\l$,
  $f\rest S\le F_\l\rest S$ implies $f\le F_\l$.

  Choose $\l$ sufficiently generic so that Theorem~\ref{thm:mult-mc-bound}
  applies for the set $S$. We claim that $f\rest S\le F_\vell\rest S$. Indeed,
  the inequality is trivial for points $p\in S$ where $V$ is singular or
  where the multiplicity is infinite, and
  for the remaining points it follows from~\eqref{eq:mult-mc-bound} and the
  simple observation
  \begin{equation}
    i(p; \mc^k(P^\l)\cdot\las(p);M_0) \le \degf_{\mc^k(P^\l)}.
  \end{equation}
\end{proof}

\begin{Rem}\label{rem:even-betti}
  In fact, since we are interested in \emph{upper bounds} for the
  \emph{Euler characteristic}, it would be reasonable to include in
  the definition of multiplicity cycles only those polar varieties
  that contribute Betti numbers of even dimension, or even include
  those that contribute Betti numbers of odd dimension with a negative
  sign. This would improve many of our multiplicity estimates roughly
  by a factor of two. We have avoided this in the present paper in
  order to simplify the notation. However,
  see~\secref{sec:improving-gr} for an illustration.
\end{Rem}

We now give estimates on the degrees of the multiplicity cycles in the
torus and affine cases. We will assume for simplicity that
$(n-1)\Delta_x\subset\Delta(P)$ where $\Delta_x$ denotes the standard
simplex in the $x$-variables. Under this assumption we have
$\Delta(P^\l)=\Delta(P)$.

The following estimates are obtained in a
straightforward manner from the corresponding propositions for polar
varieties, namely Proposition~\ref{prop:pv-deg} and
Proposition~\ref{prop:pv-deg-affine}, by noting that the equations
defining $X^r$ have Newton polygons contained in
$\Delta(P)+(r-1)\Delta(V)$.

\begin{Prop} \label{prop:mc-degree}
  Suppose that a translate of $(n-1)\Delta_x$ is contained in
  $\Delta(P)$. Then we have the bound
  \begin{equation}
    \begin{split}
      \deg\mc^k(P^\l) &\le \sum_{r=1}^{n-k} \tbinom{n}{r+k} n!
        \qmi_k(\Delta(P)+(r-1)\Delta(V)+\Delta_x)  \\
      &< 2^n n! \qmi_k(\Delta(P)+(n-k-1)\Delta(V)+\Delta_x)
    \end{split}
  \end{equation}
  where $\deg$ denotes degree in $(\C^*)^n$, $\qmi$ denotes the
  quermassintegral, and $\Delta_x$ denotes the standard simplex in the
  $x$ variables.
\end{Prop}

\begin{Prop} \label{prop:mc-degree-affine} Suppose that
  $\Delta\subset\Z_{\ge0}^n$ is a convex co-ideal containing
  $(n-1)\Delta_x$. Then we have the bound
  \begin{equation}
    \begin{split}
      \deg\mc^k(P^\l) &\le \sum_{r=1}^{n-k} \tbinom{n}{r+k} n!
        \qmi_k(\Delta(P)+(r-1)\Delta(V))  \\
      &< 2^n n! \qmi_k(\Delta(P)+(n-k-1)\Delta(V))
    \end{split}
  \end{equation}
  where $\deg$ denotes degree in $\C^n$, $\qmi$ denotes the
  quermassintegral.

  In particular, if $P$ is a polynomial of degree $d\ge n-1$ and $V$ is
  a polynomial vector field of degree $\delta$, then
  \begin{equation}
    \deg\mc^k(P^\l) < 2^n(d+(n-k-1)(\delta-1))^{n-k}
  \end{equation}
\end{Prop}

\subsection{Improving the estimates of Nesterenko, Gabrielov and Risler}
\label{sec:improving-ng}

In this section we show how Theorem~\ref{thm:mult-mc-uniform-bound} and
Proposition~\ref{prop:mc-degree-affine} imply a strengthening of the
results of Nesterenko~\cite{nesterenko:mult-estimates},
Gabrielov~\cite{gabrielov:mult} and Gabrielov and
Risler~\cite{gr:mult-c3}. We therefore restrict attention to the case
where $P,V$ given as in \eqref{eq:VP}. We fix $\l$ sufficiently
generic for the application of Theorem~\ref{thm:mult-mc-uniform-bound}.

\subsubsection{The case of a single point in arbitrary dimension}

We assume for simplicity of the formulation that $d\ge n-1$.
If $p\in\C^n$ is a non-singular point of $V$ and $\mult_p^V P<\infty$
then by Theorem~\ref{thm:mult-mc-uniform-bound} and
Proposition~\ref{prop:mc-degree-affine} we have
\begin{multline}
  \mult_p^V P \le \sum_{k=0}^{n-1} \degf_{\mc^k(P^\l)} (p)
   \le \sum_{k=0}^{n-1} \deg\mc^k(P^\l) \\
   \le 2^n \sum_{k=0}^{n-1} (d+(n-k-1)(\delta-1))^{n-k} 
   \le 2^{n+1} (d+(n-1)(\delta-1))^n 
\end{multline}
which improves the estimates of Nesterenko and Gabrielov for the
case of a single point.

\subsubsection{The case of a single point in $\C^3$}
\label{sec:improving-gr}

In~\cite{gr:mult-c3} Gabrielov and Risler considered the case $n=3$ in
detail using a different deformation technique. Their estimate, which
is the best estimate known for this particular case, is as follows
\begin{equation} \label{eq:gr-bound}
  \mult_p^V P \le d+2d(d+\delta-1)^2.
\end{equation}
A naive application of Theorem~\ref{thm:mult-mc-bound} does not yield
an improvement of this result. However, using the more refined
approach indicated in Remark~\ref{rem:even-betti} one can still obtain
an improvement using our method.

Assume for simplicity that $d\ge n-1=2$ (the remaining case $d=1$ can
be treated separately, for instance by reduction of dimension; we
leave the details for the reader). Then, in the notations
of~\secref{sec:milnor-mult} we have three Milnor fibers $F^{0,1,2}_p$
and by Remark~\ref{rem:even-betti} we are interested in an upper
bound for the sum of their even Betti numbers. Simple computations
using the corresponding polar varieties give
\begin{align*}
  b_0(F_p^0) &\le d \\
  b_0(F_p^1) &\le d(d+\delta-1) \\
  b_0(F_p^2) &\le d(d+\delta-1)(d+2\delta-2) \\
  b_2(F_p^0) &\le d(d-1)^2
\end{align*}
and accordingly,
\begin{equation}
  \mult_p^V P \le d\left[1+(d-1)^2+(d+\delta-1)(d+2\delta-1)\right]
\end{equation}
and it is a simple exercise, left for the reader, to verify that this
improves~\eqref{eq:gr-bound} for any $d,\delta$.

\subsubsection{The case of several points}

Moving now to the case of several points, let
$p_1,\ldots,p_\nu\in\C^n$ be non-singular points of $V$ and assume that
$\mult^V_{p_i}P<\infty$. Recall the notations of~\secref{sec:nesterenko}.
We consider first the case $\kappa=n$.

Let $Z$ denote any $k$-cycle in $\C^n$ and write $Z=Z_1+\ldots+Z_q$
where each $Z_i$ is a cycle supported on an irreducible variety
(possibly with a coefficient greater than 1). Then
\begin{equation}\begin{aligned}  
  \sum_{i=1}^\nu \degf_Z (p_i) &\le \sum_{j=1}^q \sum_{i=1}^\nu \degf_{Z_j} (p_i)
  \le \sum_{j=1}^q a(Z_j) \deg Z_j \\
  &\le (\max_j a(Z_j)) \sum_{j=1}^n \deg Z_j = a(Z) \deg Z
\end{aligned}\end{equation}
where $a(Z_j)$ denotes the number of points $p_i$ lying in $Z_j$, and
$a(Z)$ denotes the maximal number of points $p_i$ lying in one of the
irreducible components of $Z$. We now proceed with the multiplicity
estimate, again relying on Theorem~\ref{thm:mult-mc-uniform-bound} and
Proposition~\ref{prop:mc-degree-affine}
\begin{equation} \label{eq:our-multpoint-est}
  \begin{aligned}
    \sum_{i=1}^\nu \mult_{p_i}^V P &\le \sum_{i=1}^\nu \sum_{k=0}^{n-1} \degf_{\mc^k(P^\l)} (p_i) \\
    &\le \sum_{k=0}^{n-1} a(\mc^k(P^\l)) \deg \mc^k(P^\l) \\
    &\le 2^n \sum_{k=0}^{n-1} a(\mc^k(P^\l))
    (d+(n-k-1)(\delta-1))^{n-k}
  \end{aligned}
\end{equation}
and noting that $\mc^k(P^\l)$ does indeed have degree of
the order $O(d^{n-k})$ with respect to $d$, we obtain Nesterenko's
estimate (with improved constants).

Finally, we consider the case $\kappa<n$. That is, we now assume that
all points $p_i$ belong to a single trajectory $\gamma$ which has
transcendence degree $\kappa$. Let $Y\subset\C^n$ denote the algebraic
closure of $\gamma$. Then $\dim Y=\kappa$ and $Y$ is invariant under
the flow of $V$ (since it has a Zariski dense subset, namely
$\gamma$, which is invariant). Since the flow of $V$ maps $Y$ to
itself and maps the ambient space $\C^n$ biholomorphically to itself
(whenever defined), and since the singular part of an analytic set is
a holomorphic invariant, it follows that the singular part $\Sing V$
is invariant under the flow of $V$ as well.

We claim that the points $p_i$ belong to the smooth part of $Y$.
Indeed, suppose that some point $p_i$ belongs to $\Sing Y$. Then since
$\Sing Y$ is invariant under the flow of $V$, it follows that the germ
$\gamma_{p_i}$ is contained in $\Sing Y$. Since we assume that all
points $p_i$ belong to a single trajectory $\gamma$, by analytic
permanence it follows that $\gamma\subset\Sing Y$, contradicting our
assumption that $Y$ is the Zariski closure of $\gamma$.

One can now carry out all preceding computations in the ambient space
$Y$ instead of $\C^n$: the only assumption which is needed is the
smoothness of the ambient space at the points being considered.
Naturally, in the estimates of the degrees of the corresponding
multiplicity cycles, the degrees of the equations defining $Y$ would
play a role giving rise to existential constants as in Nesterenko's
result. However, these existential constants do not affect the
asymptotic dependence on $d$, which agrees with Nesterenko's estimate.
We omit the details of this computation.

\begin{Rem} \label{rem:small-kappa}
  In the case $\kappa=n$, the constants appearing in our result, as
  well as Nesterenko's, are explicit. In the case $\kappa<n$ the
  constants, for both proofs, depend on the algebraic complexity (for
  instance the degree) of the Zariski closure $Y$. This degree cannot
  in general be estimated in terms of $n,d,\delta$, as illustrated by
  the vector field $r x \partial_x+ s y\partial_y$ which admits a trajectory $\{x^s=y^r\}$
  of degree depending on the coefficients $r,s$.

  However, using our method one can obtain estimates with explicit
  constants --- albeit involving terms of order up to $d^n$ --- even
  when $\kappa<n$. Indeed, nowhere in the derivation
  of~\eqref{eq:our-multpoint-est} did we use the assumption
  $\kappa=n$. On the other hand, Nesterenko's approach appears to be
  dependent in a more essential way on the assumption $\kappa=n$, and
  it is not clear that it can be used to produce explicit bounds, even
  ones allowing terms of order $d^n$, when $\kappa<n$.
\end{Rem}

\subsection{Concluding remarks and some directions for future research}

Beyond the general type of multiplicity estimates considered in this
paper, many different forms have been treated in the literature. It
would be interesting to see if the methods used in this paper could
be generalized to these contexts. We list a few examples below.

Many results have been obtained for the case when the ambient manifold
is a commutative algebraic group, the vector field is an invariant
field for the group, and the set of points is a ``cube'' of a
specified dimension and length. For a survey of some of these results
and their applications in transcendental number theory
see~\cite{masser:zero-est-survey}.

Another possible generalization is for the case of analytic
trajectories at singular points of the vector field $V$.
In~\cite{nesterenko:modular-trans}, Nesterenko considers a singular
vector field satisfying the additional ``D-property''. Under this
extra assumption, Nesterenko again obtains estimates which are sharp
up to a multiplicative constant with respect to $d$. This result and
various generalizations also play an important role in transcendental
number theory.

Finally, in~\cite{gk:mult} Gabrielov and Khovanskii consider
multiplicity estimates in several dimensions. Specifically, they
consider a tuple of commuting vector fields $V_1,\ldots,V_m$ defining
an integral manifold $\cL$ of dimension $m$, and a tuple of $m$
polynomials $P_1,\ldots,P_m$. They give an estimate for the maximal
multiplicity of an isolated common zero $P_1=\cdots=P_m=0$. Our method
does not directly extend to this generality due to some technical
difficulties (specifically, the literal analog of
Lemma~\ref{lem:gab-sard} fails), but it would be interesting to check
whether similar ideas can be used to improve this result.

\appendix

\section{A compactness property for semicontinuous bounds}
\label{sec:appendix}
  
In this appendix we will assume for simplicity of the formulation that
the ambient variety $M$ is the affine space $\C^n$ or the torus
$(\C^*)^n$, though the ideas can be carried out verbatim in a much more
general context.

Recall that a function $F:M\to\N$ is said to be (algebraic)
\emph{upper semicontinuous} if the sets
$F_{\ge n}:=F^{-1}([n,\infty))$ are closed algebraic varieties for
each $n\in\N$. We will say that $F$ has \emph{complexity bounded by
  $D$} if moreover, all of these sets can be defined by equations of
degree at most $D$.

\begin{Prop} \label{prop:sc-compact}
  Let $D\in\N$ and $f:M\to\N$ an arbitrary bounded function. Then
  there exists a finite set of points $P\subset M$ such that for any
  upper semicontinuous function $F$ of complexity bounded by $D$,
  \begin{equation}
    f\rest P \le F\rest P \implies f \le F.
  \end{equation}
\end{Prop}
\begin{proof}
  Denote by $N$ an upper bound for $f$. Then $f\le F$ if and only if
  $f_{\ge i}\subset F_{\ge i}$ for $i=1,\ldots,N$. Thus it will
  suffice to construct a finite set $P_i\subset f_{\ge i}$ such that for any set $S$ of
  complexity bounded by $D$,
  \begin{equation}
    P_i \subset S \implies f_{\ge i}\subset S
  \end{equation}
  and take $P=\cup_{i=1}^N P_i$.

  Let $L$ denote the linear space of polynomials of degree bounded by
  $D$ on $M$. For any $p\in M$ let $\phi_p:L\to\C$ denote the
  functional of evaluation at $p$. Finally, for any set $P\subset M$
  denote by $L_P\subset L$ the linear subspace of polynomials which vanish at
  every point of $P$.

  We need to construct a finite set $P_i\subset f_{\ge i}$ with
  $L_{P_i}=L_{f_{\ge i}}$. This is clearly possible. Indeed,
  $L_{f_{\ge i}}$ is the kernel of the set of functionals
  $\{\phi_p:p\in f_{\ge i}\}$. Since $L_{f_{\ge i}}$ has finite
  codimension in $L$, one can choose a finite subset $P_i$ (in fact,
  of size equal to this codimension) of functionals whose kernel,
  $L_{P_i}$ agrees with $L_{f_{\ge i}}$. This concludes the proof.
\end{proof}

The proofs of the following simple lemmas are left for the reader.

\begin{Lem}\label{lem:sc-sum}
  Let $F_i,i=1,\ldots,N$ be upper semicontinuous functions with
  complexity $D_i$ and bounded by $B_i$. Then $\sum_{i=1}^N F_i$ is an
  upper semicontinuous function with complexity bounded by a number
  depending only on $D_1,\ldots,D_N$ and $B_1,\ldots,B_N$.
\end{Lem}


\begin{Lem} \label{lem:sc-irr-var}
  If $V\subset M$ is an irreducible variety of degree bounded by $d$,
  then $\degf_V$ is an upper semicontinuous function of complexity
  bounded by $d$.
\end{Lem}

For the proof of the second lemma it suffices to recall the standard
fact that a variety of degree $d$ is cut out set-theoretically by
equations of degree bounded by $d$. Finally, we have the following
simple corollary.

\begin{Cor} \label{cor:degf-complexity} Let $C$ be an algebraic cycle
  (possibly of mixed dimension) of total degree bounded by $d$. Then
  $\degf_C$ is an upper semicontinuous function of complexity bounded
  by a number $D=D(d)$ depending only on $d$.
\end{Cor}
\begin{proof}
  Indeed, $\degf_C$ is a sum of at most $\deg C$ upper semicontinuous
  functions, each of complexity and value bounded by $\deg C$
  according to Lemma~\ref{lem:sc-irr-var}. The statement then follows
  by Lemma~\ref{lem:sc-sum}.
\end{proof}

\newpage

\section{List of notations}
\label{sec:notations}

The following table lists some of the main notations used in this paper
along with a brief description and a reference for the definition where
applicable. \vspace{1cm}

\begin{tabular}{|c|p{0.5\textwidth}|c|}
  \hline
  Notation & Meaning & Definition \\ \hline

  $\mult_p^V P$ & Multiplicity of $P$ at $p$ in the direction of the vector field $V$ & \secref{sec:intro} \\
  $\Delta(P),\Delta(V)$ & Newton polytope of polynomial $P$ (resp. vector field $V$) & \secref{sec:bkt} \\
  $\Delta_x$ & Standard simplex in $x$-variables & \secref{sec:bkt} \\
  $\degf_C$ & Degree function of the cycle $C$ & Definition~\ref{def:degf} \\
  $i(Z;V\cdot W; M)$ & The multiplicity of $Z$ as a component of the intersection $V\cdot W$ & \secref{sec:cycles} \\
  $\deg C$ & The degree of the cycle $C$ & \secref{sec:cycles} \\
  $V(\cdots)$ & Mixed volume & \secref{sec:bkt} \\
  $\qmi_j(\Delta)$ & The $j$-th simplicial quermassintegral of $\Delta$ & Equation~\eqref{eq:qmi-def} \\
  $\las(x)$ & An affine space of codimension $k$, through the point $x$, in the direction specified
  by $\vell$ & Equation~\eqref{eq:las-def} \\
  $\pv(X)$ & The $k$-th polar variety associated to a deformation $X$ & Definition~\ref{def:polar} \\
  $\bpv(X^t),\tpv(X^r)$ & The refined polar variety (resp. its open dense subset) & Definition~\ref{eq:bpv-def} \\
  $\Sigma(X^r)$ & The (fiberwise) singular locus of $X^r$ & Equation~\eqref{eq:sigma-def} \\
  $b_k(\cdot)$ & The $k$-th Betti number & \\
  $F_p^r$ & The Milnor fiber of the family $X^r$ at $p$ & Definition~\ref{def:milnor-fiber} \\
  $\mc^k(P^\l)$ & The multiplicity cycle associated to the deformation $P^\l$ & Definition~\ref{def:mc} \\
  \hline
\end{tabular}

\newpage

\let\~=\tildeaccent
\bibliographystyle{plain}
\bibliography{refs}

\end{document}